\documentclass{amsart}
\usepackage[T1]{fontenc}
\usepackage{lmodern}
\usepackage[latin1]{inputenc}
\usepackage[english]{babel}

\usepackage{amsmath, amsthm, amssymb}
\usepackage{url} 
\usepackage{latexsym} 
\usepackage{graphicx} 
\usepackage[all]{xy} 
\usepackage{color} 
\usepackage{enumerate} 
\usepackage{yhmath} 

\usepackage[left=3.5cm,right=3cm,top=3.3cm,bottom=2.5cm]{geometry}

\newcommand\ZZ{\mathbb{Z}}
\newcommand\QQ{\mathbb{Q}}
\newcommand\RR{\mathbb{R}}

\newcommand{\beq}[1]{\begin{equation}\label{#1}} 
\newcommand{\eeq}{\end{equation}}                
\newcommand{\beqn}{\begin{equation*}}
\newcommand{\eeqn}{\end{equation*}}

\newcommand{\pp}[1]{{\left({#1}\right)}}  

\newcommand{\Abs}[1]{{\left|{#1}\right|}}  
\renewcommand{\int}[1]{{\left\lceil{#1}\right\rceil}}  

\newtheorem{thm}{{Theorem}}

\newtheorem{lema}{{Lemma}}

\title{Concentration of points on Modular Quadratic Forms}

\author[A. Zumalac\'{a}rregui]{Ana Zumalac\'{a}rregui}
\address{Departamento de Matem\'aticas, Universidad Aut\'onoma de Madrid and Instituto de Ciencias Matem\'aticas (CSIC-UAM-UC3M-UCM), 28049 Madrid, Spain}
\email{ana.zumalacarregui@uam.es}

\begin{document}

\begin{abstract}
Let $Q(x,y)$ be a quadratic form with discriminant $D\neq 0$. We obtain non trivial upper bound estimates
for the number of solutions of the congruence $Q(x,y)\equiv
\lambda \pmod{p}$, where $p$ is a prime and $x,y$ lie in certain
intervals of length $M$, under the assumption that $Q(x,y)-\lambda$ is an absolutely irreducible polynomial modulo $p$. In particular we prove that the number of solutions to this congruence is $M^{o(1)}$ when $M\ll p^{1/4}$.
These estimates generalize a previous result by Cilleruelo and
Garaev on the particular congruence $xy\equiv
\lambda \pmod{p}$.
\end{abstract}
\thanks{The author was supported by Departamento de Matem\'aticas, Universidad Autónoma de Madrid, Spain.}
\date{January 31, 2011}

\maketitle 

\section{Introduction}
Let $Q(x,y)$ be a quadratic form with discriminant $D\neq 0$. For
any odd prime $p$ and $\lambda\in \ZZ$, we consider the congruence
\beq{mod_quadratic} Q(x,y)\equiv \lambda \pmod {p}
\qquad\left\{\begin{array}{l} K+1\leq x\leq K+M,\\ L+1\leq y\leq
L+M,\end{array}\right. \eeq for arbitrary values of $K,L$ and $M$.
We denote by $I_Q(M;K,L)$ the number of solutions to
\eqref{mod_quadratic}. 

\smallskip

It follows from \cite{Shpa-Voloch,Zheng} that if the quadratic form $Q(x,y)-\lambda$ is absolutely irreducible modulo $p$, one can derive from the Bombieri bound \cite{Bom} that \beq{trigonometric_estimate}
I_Q(M;K,L)=\frac{M^2}{p}+O(p^{1/2}\log^2p).
\eeq 

Whenever $M$ is small, say $M\ll p^{1/2}\log^2p$, this estimate provides an upper bound which is worse than the trivial estimate $I_Q(M;K,L)\leq 2M$
(for every $x$ in the range we have a second degree polynomial in
$y$ with no more than two solutions). \

\smallskip

In the special case $Q(x,y)=xy$ and $(\lambda,p)=1$, Chan and Shparlinsky \cite{Chan}
used sum product estimates to obtain a non trivial estimate \beqn
I(M;K,L)\ll M^2/p+M^{1-\eta}, \eeqn for some $\eta>0$.
Cilleruelo and Garaev \cite{Ci-Garaev}, using a different method,
improved this estimate: \beqn I(M;K,L)\ll
\pp{M^{4/3}p^{-1/3}+1}M^{o(1)}.\eeqn The aim of this work is to
generalize Cilleruelo and Garaev's estimate to any non-degenerate
quadratic form.

\begin{thm}\label{Thm}
 Let $Q(x,y)$ be a quadratic form defined over $\ZZ$, with discriminant $D\neq 0$. For any  prime $p$ and $\lambda\in \ZZ$ such that $Q(x,y)-\lambda$ is absolutely irreducible modulo $p$, we have
\beqn I_Q(M;K,L)\ll \pp{M^{4/3}p^{-1/3}+1}M^{o(1)}. \eeqn
\end{thm}

This estimate is non trivial when $M=o(p)$ and better than
\eqref{trigonometric_estimate} whenever $M\ll p^{5/8} $.
Furthermore, when $M\ll p^{1/4}$ Theorem \ref{Thm} gives
$I_Q(M;K,L)=M^{o(1)}$, which is sharp. Probably the last estimate
also holds for $M\ll p^{1/2}$, but it seems to be a difficult
problem.

\smallskip

Note that if
\beqn
Q(x,y)-\lambda\equiv q_1(x,y)q_2(x,y)\pmod{p},
\eeqn
for some linear polynomials $q_i(x,y)\in \ZZ[x,y]$, we have that solutions in \eqref{mod_quadratic} will correspond to solutions of the linear equations $q_i(x,y)\equiv 0 \pmod {p}$ and we could have $\gg M$ different solutions. The condition of irreducibility is required to avoid this situation.

\smallskip

Observe that the condition $D\neq 0$ restrict ourselves to the study
of ellipses and hyperbolas. The given upper bound cannot be applied
to quadratic forms with discriminant $D=0$. For example the number
of solutions to \eqref{mod_quadratic} when $Q(x,y)=y-x^2$ is $\asymp
M^{1/2}$.

\section{Proof of Theorem \ref{Thm}}

The following lemmas will be required during our proof. These results will give us useful upper bounds over the number of lattice points in arcs of certain length on conics.

\begin{lema}\label{small_arc}
 Let $D\neq 0,1$ be a fixed square-free integer. On the conic $x^2-Dy^2=n$ an arc of length $n^{1/6}$ contains, at most, two lattice points.
\end{lema}

This lemma is a particular case of Theorem 1.2 in \cite{Ci-Urroz}.

\begin{lema}\label{big_arc}
 Let $D\neq 0,1$ be a fixed square-free integer. If $n=M^{O(1)}$, on the conic $x^2-Dy^2=n$ an arc of length $M^{O(1)}$ contains, at most, $M^{o(1)}$ lattice points.
\end{lema}

\begin{proof}
 This result is a variant of Lemma 4 in \cite{Ci-Garaev}, where the conclusion was proved when $1\leq x,y\leq M^{O(1)}$, (see Lemma 3.5 \cite{Vau-Wool} for a more general result).

\smallskip

If $D$ is negative, the result is contained in Lemma 4 in
\cite{Ci-Garaev} since it is clear that $1\le x,y\ll \sqrt
n=M^{O(1)}$. We must study though the case where $D$ is positive.

\smallskip

By symmetry we can consider only those arcs in the first quadrant, since any non-negative lattice point $(x,y)$ will lead us to no more than four lattice points $(\pm x,\pm y)$. Let $(u_0,v_0)$ be the minimal non-negative solution to the Pell's
equation $x^2-Dy^2=1$, and $\xi=u_0-\sqrt{D}v_0$ its related
fundamental unit in the ring of integers of $\QQ\pp{\sqrt{D}}$. Suppose
that $(x_0,y_0)$ is a positive solution to
$x_0^2-Dy_0^{2}=n$ that lies in our
initial arc and let $t\in\RR$ be the solution to
\beqn
(x_0+\sqrt{D}y_0)\xi^t=(x_0-\sqrt{D}y_0)\xi^{-t}.
\eeqn
Then for $m=[t]$, we have $(x_0+\sqrt{D}y_0)\xi^m=x_1+\sqrt{D}y_1\asymp
\sqrt{n}$. This means that each solution in our initial arc corresponds
to a `primitive' solution lying in an arc of length $\ll \sqrt{n}$.
Conversely, solutions in an arc of length $\ll \sqrt{n}$ can be
taken to larger arcs by multiplying by powers of $\xi^{-1}$. Since
our initial interval has length $M^{O(1)}$ there will be no more than
$O(\log M)$ powers connected to each primitive solution. The term
$O(\log M)$ is absorbed by $M^{o(1)}$. \

On the other hand, we know by Lemma 4 in \cite{Ci-Garaev} that the number of lattice points in an arc of length $O(\sqrt{n})$ is $M^{o(1)}$.
It follows that the number of solutions in the original arc will be bounded by $M^{o(1)}$.
\end{proof}

We are now in conditions to start the proof of Theorem \ref{Thm}.

\begin{proof}
Let $Q(x,y)=ax^2+bxy+cy^2+dx+ey+f$ be a quadratic form with
integer coefficients and discriminant $D=b^2-4ac\neq 0$. Whenever $a=c=0$, the congruence in \eqref{mod_quadratic} can be written in the form $XY\equiv \mu \pmod{p}$, where $X=bx+e$, $Y=by+d$ and $\mu=b\lambda-(ed+bf)$. This case was already studied in \cite{Ci-Garaev},  but one extra condition was required: $\mu$ must be coprime with $p$ or, equivalently, $XY-\mu$ must be absolutely irreducible modulo $p$. 

\smallskip

If $a\neq 0$ the congruence in \eqref{mod_quadratic} can be written as
\beqn
X^2-DY^2\equiv \mu \pmod{p},
\eeqn
where $X=Dy+2(ae-db)$, $Y=2ax+by+d$ and $\mu=4aD\lambda-D(4af-d^2)+4a(ae-db)$. The case $a=0$ and $c\neq 0$
follows by exchanging $x$ for $y$ in the previous argument (and so $c,e$ will be the coefficients of $x^2$ and $x$ instead of $a,d$). Our new variables $X,Y$ lie in intervals of length $\ll M$. Specifically $X$ lies in an interval of length $DM$ and $Y$ in an interval of length $\pp{2\Abs{a}+\Abs{b}}M$.

\smallskip

We also can assume that $p>D$. Since $D\neq 0$, different original
solutions will lead us to a different solution.

\smallskip

These observations allow us to bound the number of solutions to
\eqref{mod_quadratic} by the number of solutions of the congruence
\beqn x^2-Dy^2\equiv \mu \pmod{p}, \eeqn where $x,y$ lie in two
intervals of length $\ll M$.

\smallskip

Without loss of generality we can assume that $D$ is square-free. Otherwise $D=D_1k^2$, for some square-free integer $D_1$, and solutions $(x,y)$ of our
 equation would lead us to solutions $(x,ky)$ of $x^2-D_1(ky)^2\equiv \mu\pmod{p}$, where $ky$ would lie in some interval of length $\ll M$. The case $D=1$ corresponds to the problem $x^2-y^2=UV\equiv \mu\pmod{p}$, where $U=(x+y)$ and $V=(x-y)$ still lie in some intervals of length $\ll M$ and $(\mu,p)=1$, otherwise $UV-\mu$ will be reducible modulo $p$. Once more this case was already studied in \cite{Ci-Garaev}.

\smallskip

By the previous arguments it is enough to prove the result for
\beq{mod_hyp_ell}
x^2-Dy^2\equiv \lambda\pmod{p}, \qquad\left\{\begin{array}{l} K+1\leq x\leq K+M,\\ L+1\leq y\leq L+M,\end{array}\right.
\eeq
where $D$ is some square-free integer $\neq 0,1$ and $\lambda\in \ZZ$.

\smallskip

This equation is equivalent to \beqn
\pp{x^2+2Kx}-D\pp{y^2+2Ly}\equiv \mu\pmod {p}, \qquad 1\leq x,y\leq
M, \eeqn where $\mu=\lambda-(K^2-DL^2)$. By the pigeon hole principle
we have that for every positive integer $T<p$, there exists a
positive integer $t<T^2$ such that $tK\equiv k_0\pmod{p}$ and
$tL\equiv \ell_0\pmod {p}$ with $\Abs{k_0},\Abs{\ell_0}<p/T$. Thus
we can always rewrite the equation \eqref{mod_hyp_ell} as \beqn
\pp{tx^2+2k_0x}-D\pp{ty^2+2\ell_0y}\equiv \mu_0\pmod {p},\qquad
1\leq x,y\leq M, \eeqn where $\Abs{\mu_0}<p/2$. This modular
equation lead us to the following Diophantine equation
\beq{dioph_hyperbola}
\pp{tx^2+2k_0x}-D\pp{ty^2+2\ell_0y}=\mu_0+pz,\qquad  1\leq x,y\leq
M,\ z\in\ZZ, \eeq where $z$ must satisfy \beqn
\Abs{z}=\Abs{\frac{\pp{tx^2+2k_0x}-D\pp{ty^2+2\ell_0y}-\mu_0}{p}}<\frac{(1+\Abs{D})T^2M^2}{p}+\frac{2(1+\Abs{D})M}{T}+\frac{1}{2}.
\eeqn For each integer $z$ on the previous range the equation
defined in \eqref{dioph_hyperbola} is equivalent to:
\beq{hyperbola_ellipse} (tx+k_0)^2-D(ty+\ell_0)^2= n_z,\qquad 1\leq
x,y\leq M, \eeq where $n_z=t(\mu_0+pz)+(k_0^2-D\ell_0^2)$. We will
now study the number of solutions in terms of $n_z$.

\smallskip

If $n_z=0$, since $D$ is not a square, we have that $
tx+k_0=ty+\ell_0=0$ and there is at most one solution $(x,y)$.

\smallskip

Let now focus on the case $n_z\neq 0$. We will split the problem in
two different cases, depending on how big $M$ is compared to $p$.
\begin{itemize}
    \item Case ${M < \frac{p^{1/4}}{4\sqrt[4]{\pp{1+\Abs{D}}^3}}}$. In this case we take $T=8\pp{1+\Abs{D}}M$ in order to get $\Abs{z}<1$. Therefore it suffices to study solutions of
\beqn (tx+k_0)^2-D(ty+\ell_0)^2=n_0,\qquad 1\leq x,y\leq M. \eeqn

\smallskip

If $n_0 > 2^{48}(1+\Abs{D})^{12}M^{18}$, the integers $\Abs{tx+k_0}$
and $\Abs{ty+\ell_0}$ will lie in two intervals of length
$T^2M=2^6(1+\Abs{D})^2M^3$ and solutions to \eqref{mod_hyp_ell}
will come from lattice points in an arc of length smaller than
$2^8(1+\Abs{D})^2M^3<n_0^{1/6}$ (by hypothesis). From Lemma
\ref{small_arc} it follows that there will be no more than two
lattice points in such an arc.

\smallskip

If $n_0\leq 2^{48}(1+\Abs{D})^{12}M^{18}$, Lemma \ref{big_arc}
assures that the number of solutions will be $M^{o(1)}$.

\smallskip

    \item Case ${M\geq \frac{p^{1/4}}{4\sqrt[4]{(1+\Abs{D})^3}}}$. In this case we take $T=(p/M)^{1/3}$ and hence ${\Abs{z}\ll\frac{M^{4/3}}{p^{1/3}}}$.\\ Since $n_z=t(\mu_0+pz)+(k_0^2-D\ell_0^2)\ll p^2\ll M^8$ we can apply Lemma \ref{big_arc} to conclude that for every $z$ in the range above there will be $M^{o(1)}$ solutions to its related Diophantine equation.
\end{itemize}
We have proved that in all cases, the number of solutions to \eqref{hyperbola_ellipse} is $M^{o(1)}$ for each
$n_z$. On the other hand, the number of possible
values of $z$ is $O(M^{4/3}p^{-1/3}+1)$. It follows that \beqn
I_Q(M;K,L)\, \ll\left (M^{4/3}p^{-1/3}+1\right )M^{o(1)}. \eeqn
\end{proof}

\section*{Acknowledgments}

I would like to thank J. Cilleruelo and M. Garaev for their advice and helpful suggestions in the preparation of this paper and the referee for his  valuable comments.


\begin{thebibliography}{0}

\bibitem{Bom} E. Bombieri, On exponential sums in finite fields, \textit{Amer. J. Math.} \textbf{88} (1966), 71--105.

\bibitem{Chan} T. H. Chan and I. Shparlinski, On the concentration of points on modular hyperbolas and exponential curves, \textit{Acta Arithmetica} \textbf{142} (2010), 59--66.

\bibitem{Ci-Garaev} J. Cilleruelo and M. Garaev, Concentration of points on two and three dimensional modular hyperbolas and applications,
\textit{preprint}: arXiv:1007.1526v2 (12 pages), 2010.

\bibitem{Ci-Urroz} J. Cilleruelo and J. Jim\'{e}nez-Urroz, Divisors in a Dedekind domain, \textit{Acta Aritmetica} \textbf{85} (1998), 229--233.

\bibitem{Vau-Wool} R. C. Vaughan and T. D. Wooley, Further improvement in Wairing's problem, \textit{Acta Math.} \textbf{174} (1995), 147--240. 

\bibitem{Shpa-Voloch} I. E. Shparlinski and J. F. Voloch, Visible points on curves over finite fields, \textit{Bull. Polish Acad. Sci. Math.} \textbf{55} (2007), 193--199.

\bibitem{Zheng} Z. Zheng, The distribution of zeros of an irreducible curve over a finite field, \textit{J. Number Theory} \textbf{59} (1996), 106--118. 


\end{thebibliography}
\end{document}